\documentclass{elsarticle}

\usepackage{amsmath, amsfonts, amssymb, amsthm, xparse}
\numberwithin{equation}{section}

\newtheorem{lemma}{Lemma}

\newtheorem{theorem}{Theorem}
\newtheorem{corollary}{Corollary}[section]
\newtheorem{referencetheorem}{Theorem}

\newtheorem{referencelemma}[referencetheorem]{Lemma}

\theoremstyle{definition}\newtheorem*{definition}{Definition}

\newcommand{\D}{\mathcal{D}}
\newcommand{\n}[1]{\left\| #1 \right\|}
\newcommand{\br}[1]{\left( #1 \right)}

\begin{document}

\begin{frontmatter}
\title{On the generalized Approximate Weak Chebyshev Greedy Algorithm}
\author{A.\,V.~Dereventsov}
\ead{dereventsov@gmail.com}
\address{1523 Greene street, 
		University of South Carolina, 
		Columbia, SC 29208, USA}

\begin{abstract}
The Weak Chebyshev Greedy Algorithm (WCGA) is defined for any Banach space $X$ and 
a dictionary $\D$, and provides nonlinear $n$-term approximation 
for a given element $f \in X$ with respect to $\D$.
In this paper we study the generalized  Approximate Weak Chebyshev Greedy 
Algorithm (gAWCGA) --- a modification of the WCGA in which we are 
allowed to calculate $n$-term approximation with relative and 
absolute errors in computing a norming functional, 
an element of best approximation, and an approximant. 
Such permission is natural for the numerical 
applications and simplifies realization of the algorithm.
We obtain conditions that are sufficient for the convergence of the 
gAWCGA for any element of a uniformly smooth Banach space, 
and show that they are necessary in the class of uniformly smooth 
Banach spaces with the modulus of smoothness of non-trivial power 
type (e.g. $L_p$ spaces for $1 < p < \infty$).
In particular, we show that if all the errors are from $\ell_1$ 
space then the conditions for the convergence of the gAWCGA 
are the same as for the WCGA.
We also construct an example of a smooth Banach space
in which the algorithm diverges for a dictionary and an element,
thus showing that the smoothness of a space is not
sufficient for the convergence of the WCGA.
\end{abstract}


\end{frontmatter}

\section{Introduction}\label{introduction}

This paper is devoted to the problem of greedy approximation in 
Banach spaces. 
We consider here the Weak Chebyshev Greedy Algorithm (WCGA), which was 
studied by V.\,N.~Temlyakov (see, for instance,
~\cite{temlyakov_gabs},~\cite{temlyakov_ga}).
The WCGA is defined for any Banach space, and provides nonlinear 
$n$-term approximations of a given element of the space with respect 
to a fixed set of elements.
For the purpose of numerical applications it seems logical to 
allow the steps of the WCGA to be calculated not exactly, 
but with some inaccuracies. 
We note that such approach was used for other types of greedy algorithms
(e.g. see~\cite{gribonval_nielsen_awga} and~\cite{galatenko_livshitz_gawga}).
If the reader would like to know more about other types
of greedy approximations, he may review survey 
papers~\cite{temlyakov_konyagin_garbgms},
\cite{temlyakov_garb}, and~\cite{temlyakov_gabs2}.

The modification of the WCGA with relative errors in computing 
the steps of the algorithm, the Approximate Weak Chebyshev Algorithm
(AWCGA), was studied in~\cite{temlyakov_gtabsa} and~\cite{dereventsov_awcgausbs}.
In this paper we study another modification of the WCGA, in which we
are allowed to make both absolute and relative errors on 
every step of the algorithm.
Similar to the terminology proposed in~\cite{galatenko_livshitz_gawga}
we call this modification the generalized Approximate Weak Chebyshev 
Algorithm (gAWCGA).

Recall that a dictionary is such a set $\D$ of elements of a real Banach space $X$ 
that $\overline{\text{span}}\, \D = X$ and elements of $\D$ are normalized, i.e.
$\n{g} = 1$ for any $g \in \D$. For convenience we assume that all
dictionaries are symmetric, i.e. if $g  \in \D$ then $-g \in \D$.
By $A_1(\D)$ we denote the closure of the convex hull of a dictionary $\D$,
and by $A_0(\D)$ we denote all the linear combinations of the elements
of a dictionary $\D$.

We define the following sequences, which represent 
the inaccuracies in calculating the steps of the algorithm.
A weakness sequence $\{(t_n, t'_n)\}_{n=1}^{\infty}$
is such that $0 \leq t_n \leq 1$ and $t'_n \geq 0$ for all $n \geq 1$. 
A perturbation sequence $\{(\delta_n, \delta'_n)\}_{n=0}^{\infty}$
is such that $\delta_n \geq 0$ and $\delta'_n \geq 0$ for all $n \geq 0$.
An error sequence $\{(\eta_n, \eta'_n)\}_{n=1}^{\infty}$
is such that $\eta_n \geq 0$ and $\eta'_n \geq 0$ for all $n \geq 1$.
By $\eta_0$ and $\eta'_0$ we denote the least upper bounds of the
sequences $\{\eta_n\}_{n=1}^\infty$ and $\{\eta'_n\}_{n=1}^\infty$ respectively.

For a Banach space $X$, a dictionary $\D$, and an element $f \in X$,
the generalized Approximate Weak Chebyshev Greedy Algorithm with 
a weakness sequence $\{(t_n, t'_n)\}_{n=1}^{\infty}$, 
a perturbation sequence $\{(\delta_n, \delta'_n)\}_{n=0}^{\infty}$,
and an error sequence $\{(\eta_n, \eta'_n)\}_{n=1}^{\infty}$ 
is defined as follows.

\begin{definition}{{\bf(gAWCGA)}}
Set $f_0 = f$ and for each $n \geq 1$
\begin{enumerate}
\item take any functional $F_{n-1}$ satisfying
\begin{equation}\label{gawcga_functional}
\n{F_{n-1}} \leq 1
\quad\text{and}\quad
F_{n-1}(f_{n-1}) \geq (1 - \delta_{n-1}) \n{f_{n-1}} - \delta'_{n-1},
\end{equation}
\item choose any element $\phi_n \in \D$ such that
\begin{equation}\label{gawcga_approximant}
F_{n-1}(\phi_n) \geq t_n \sup\limits_{g \in \D} F_{n-1}(g) - t'_n,
\end{equation}
\item for $\Phi_n = \text{span}\{\phi_j\}_{j=1}^{n}$ 
denote $E_n = \inf\limits_{G\in\Phi_n}\n{f - G}$
and find any $G_n \in \Phi_n$ satisfying
\begin{equation}\label{gawcga_approximation}
\n{f - G_n} \leq (1 + \eta_n) E_n + \eta'_n,
\end{equation}
\item define $G_n$ to be the $n$-th approximation 
and set $f_n = f - G_n$ to be the $n$-th remainder.
\end{enumerate}
\end{definition}

Note that if for every $n \geq 1$ either $t_n < 1$ or $t'_n > 0$ 
then for any Banach space $X$, any dictionary $\D$, and any element $f \in X$
there exists a possible realization of the algorithm.
We say that the gAWCGA of $f$ converges if every realization provides a 
sequence of approximations $\{G_n\}_{n=1}^\infty$ that converges to $f$.
Conversely, we say that the gAWCGA diverges if there exists such 
a realization that $G_n \not\to f$ as $n \to \infty$.

Note also that if $t'_n = \delta_{n-1} = \delta'_{n-1} = \eta_n = \eta'_n = 0$ 
for all $n \geq 1$ then the gAWCGA coincides with the WCGA 
which was studied in~\cite{temlyakov_gabs}
and~\cite{dilworth_kutzarova_temlyakov_csgabs}.
In the case $t'_n = \delta'_{n-1} = \eta'_n = 0$ the gAWCGA coincides
with the AWCGA which was studied 
in~\cite{temlyakov_gtabsa} and~\cite{dereventsov_awcgausbs}.

One of the goals of this paper is to investigate the behavior of the gAWCGA
in a uniformly smooth Banach space $X$ and to 
obtain conditions on the weakness, perturbation, 
and error sequences that guarantee the convergence
of the gAWCGA for all dictionaries $\D \subset X$ and all elements $f \in X$.
In section~\ref{convergence_gawcga_uniformly_smooth_spaces}
we state the sufficient conditions for the convergence of the gAWCGA
with arbitrary sequences 
$\{t'_n\}_{n=1}^\infty$, $\{\delta'_n\}_{n=0}^\infty$, and 
$\{\eta'_n\}_{n=1}^\infty$
in a uniformly smooth Banach space $X$, and show that they are
the necessary ones if $X$ has a modulus of smoothness of 
non-trivial power type.

We understand the necessity of conditions in the following way:
if at least one of the stated conditions does not hold,
one can find a uniformly smooth Banach space $X$, 
a dictionary $\D$, and an element $f \in X$ such 
that the gAWCGA of $f$ with the given weakness, perturbation, and 
error sequences, diverges.
We note that in our case such a Banach space and dictionary need not
be complicated. In fact, we demonstrate that an 
example of the divergent gAWCGA can be found even in 
$\ell_p$ space with the canonical basis as a dictionary.

In section~\ref{proofs} we prove theorems stated in 
section~\ref{convergence_gawcga_uniformly_smooth_spaces}.
We note that while we are interested in the question of strong convergence 
of the WCGA and its modifications, the more general setting was 
considered in~\cite{dilworth_kutzarova_temlyakov_csgabs}.

Another goal of this paper is to discuss the restrictions imposed
on a Banach space $X$ that are required for the convergence of the WCGA.
It is known (see~\cite{temlyakov_gabs}) that the WCGA with
a constant weakness sequence $0 < t \leq 1$ 
(denoted further as WCGA($t$)) converges in
all uniformly smooth Banach spaces for all dictionaries and 
elements of the space. However the uniform smoothness is not necessary:
it is shown in~\cite{dilworth_kutzarova_temlyakov_csgabs} 
that every separable reflexive Banach space $X$ admits an
equivalent norm for which the WCGA($t$) converges for any dictionary $\D$ 
and any element $f \in X$. Furthermore, one can find a separable reflexive 
Banach space that does not admit an equivalent uniformly smooth norm 
(e.g. see~\cite{beauzamy_ibstg}). 
Thus, the condition of uniform smoothness of a space can be
weakened. In particular, it is shown in~\cite{dilworth_kutzarova_temlyakov_csgabs} 
that if a reflexive Banach space $X$ has the Kadec-Klee property
and a Fr\'echet differentiable norm, then the WCGA($t$) converges 
for any dictionary $\D$ and any element $f \in X$.

On the other hand, it is shown in~\cite{dubinin_gaa} that the smoothness
of the space is equivalent to a decrease of norms of the 
remainders of the WCGA for any dictionary $\D$ and any element $f \in X$. 
Thus, smoothness of a space is necessary for the
convergence of the algorithm and it would be natural to assume that 
smoothness might also be a sufficient condition.
In section~\ref{smoothness_non-sufficiency_wcga} we refute
this hypothesis by demonstrating an example of a smooth Banach space,
a dictionary, and an element, for which the WCGA diverges.
To construct the desired Banach space, we adopt the technique 
that was used in~\cite{donahue_darken_gurvits_sontag_rcanhs} 
to prove the necessity of smoothness of a space for the 
convergence of the incremental approximation.
Namely, we re-norm $\ell_1$ space by introducing the sequence of 
recursively defined semi-norms $\{\vartheta_n\}_{n=1}^\infty$, 
each of which is the $\ell_{p_n}$-norm
of the previously calculated semi-norm $\vartheta_{n-1}$ 
and the $n$-th coordinate of the element, where the sequence
$\{p_n\}_{n=1}^\infty$ decreases to $1$ sufficiently fast.
The reason for such a complicated approach is that the constructed space 
must be smooth but not uniformly smooth, which is already
a non-trivial task.
We note that an analogous space was used in~\cite{livshitz_cgabs}
to prove the insufficiency of smoothness of a space for the 
convergence of the X-Greedy Algorithm.

\section{Convergence of the gAWCGA in uniformly smooth Banach spaces}
\label{convergence_gawcga_uniformly_smooth_spaces}

We begin this section by recalling a few definitions.
A functional $F_x$ is a norming functional of a non-zero element $x$ 
of a Banach space $X$ if $\n{F_x} = 1$ and $F_x(x) = \n{x}$.
A Banach space $X$ is smooth if for any non-zero element $x \in X$
there exists a unique norming functional $F_x$.

For a Banach space $X$ the modulus of smoothness $\rho(u)$ is defined by
\begin{equation}\label{modulusofsmoothness}
\rho(u) = \sup_{\n{x}=\n{y}=1} \frac{\n{x+uy} + \n{x-uy}}{2} - 1.
\end{equation}
Note that the modulus of smoothness is an even and convex function and, 
therefore, $\rho(u)$ is non-decreasing on $(0, \infty)$.
A Banach space is uniformly smooth if $\rho(u) = o(u)$ as $u \to 0$.
We say that the modulus of smoothness $\rho(u)$ is of power type 
$1 \leq q \leq 2$ if $\rho(u) \leq \gamma u^q$ for some $\gamma > 0$.
It is easy to see that any Banach space has a modulus of smoothness 
of power type $1$ and that any Hilbert space has a modulus of smoothness 
of power type $2$.
Denote by $\mathcal{P}_q$ the class of all uniformly smooth Banach spaces 
with the modulus of smoothness of nontrivial power type $1 < q \leq 2$.
In particular, it is known 
(see Lemma~B.1 from~\cite{donahue_darken_gurvits_sontag_rcanhs})
that the modulus of smoothness $\rho_p(u)$ of $L_p$ space satisfies
\[
\rho_p(u) \leq \left\{
\begin{array}{ll}
\frac{1}{p} \, u^p & 1 < p \leq 2 \\
\frac{p-1}{2} \, u^2 & 2 \leq p < \infty
\end{array}
\right.,
\]
hence $L_p \in \mathcal{P}_q$, where $q = \min\{p; 2\}$
for any $1 < p < \infty$.

For a weakness sequence $\{t_n\}_{n=1}^{\infty}$ 
and a number $0 < \theta \leq 1/2$ we define a sequence of positive numbers 
$\{\xi_n\}_{n=1}^\infty$ which satisfy the equality 
$\rho(\xi_n) = \theta t_n \xi_n$ for each $n \geq 1$.
It is shown in~\cite{temlyakov_gabs} that if a Banach space is uniformly smooth 
then for any $0 < \theta \leq 1/2$ the sequence $\{\xi_n\}_{n=1}^\infty$ 
exists and is uniquely determined by the sequence $\{t_n\}_{n=1}^\infty$.

We now state some known results concerning the convergence of the WCGA 
and its modifications in arbitrary uniformly smooth Banach spaces.
The first result states the sufficient conditions for the convergence of the 
WCGA (see Theorem~2.1 from~\cite{temlyakov_gabs}).
\begin{referencetheorem}
The WCGA with a weakness sequence $\{t_n\}_{n=1}^\infty$ converges 
for any uniformly smooth Banach space $X$, any dictionary $\D$,
any element $f \in X$ if for any $0 < \theta \leq 1/2$
\[
\sum_{n=1}^\infty t_n \xi_n = \infty.
\]
\end{referencetheorem}

The next theorem gives the sufficient conditions for the convergence of the 
AWCGA (see Theorem~2.2 from~\cite{temlyakov_gtabsa}).
\begin{referencetheorem}
The AWCGA with a weakness sequence $\{t_n\}_{n=1}^\infty$, 
a perturbation sequence $\{\delta_n\}_{n=0}^{\infty}$,
and an error sequence $\{\eta_n\}_{n=1}^{\infty}$
converges for any uniformly smooth Banach space $X$, any dictionary $\D$, 
and any element $f \in X$ if 
\[
\eta_0 = \sup_{n \geq 1} \eta_n < \infty,
\]
and if for any $0 < \theta \leq 1/2$ the following conditions hold:
\begin{align*}
&\sum_{n=1}^\infty t_n \xi_n = \infty,
\\
&\delta_n = o(t_n \xi_n),
\\
&\eta_n = o(t_n \xi_n).
\end{align*}
\end{referencetheorem}

We will prove the following theorem that states that the similar result holds for 
the convergence of the gAWCGA with somewhat weaker restrictions on the approximation 
parameters. Specifically, we require the parameters to be sufficiently small only 
along some increasing sequence of natural numbers $\{n_k\}_{k=1}^\infty$.

\begin{theorem}\label{gawcga_convergence}
The gAWCGA with a weakness sequence $\{(t_n, t'_n)\}_{n=1}^\infty$, 
a perturbation sequence $\{(\delta_n, \delta'_n)\}_{n=0}^{\infty}$,
and an error sequence $\{(\eta_n, \eta'_n)\}_{n=1}^{\infty}$
converges for any uniformly smooth Banach space $X$, any dictionary $\D$, 
and any element $f \in X$ if 
\begin{equation}
\label{gawcga_eta_boundedness}
\eta_0 = \sup_{n \geq 1} \eta_n < \infty, \ \ 
\lim_{n\to\infty}\eta'_n = 0, 
\end{equation}
and if there exists a subsequence $\{n_k\}_{k=1}^{\infty}$
such that for any $0 < \theta \leq 1/2$ the following conditions hold:
\begin{align}
\label{gawcga_sum_t_xi=infty}
&\sum_{k=1}^\infty t_{n_k+1} \xi_{n_k+1} = \infty,
\\
\label{gawcga_t'=o(t)}
&t'_{n_k+1} = o(t_{n_k+1}),
\\
\label{gawcga_delta=o(t_xi)}
&\delta_{n_k} = o(t_{n_k+1} \xi_{n_k+1}),
\\
\label{gawcga_delta'=o(t_xi)}
&\delta'_{n_k} = o(t_{n_k+1} \xi_{n_k+1}),
\\
\label{gawcga_eta=o(t_xi)}
&\eta_{n_k} = o(t_{n_k+1} \xi_{n_k+1}),
\\
\label{gawcga_eta'=o(t_xi)}
&\eta'_{n_k} = o(t_{n_k+1} \xi_{n_k+1}).
\end{align}
\end{theorem}

If modulus of smoothness of a space is of a nontrivial power type
the previous theorems can be rewritten in a form that states
necessary and sufficient conditions for the convergence.
The following result is Corollary~2.1 from~\cite{temlyakov_gabs}.
\begin{referencetheorem}
The WCGA with a weakness sequence $\{t_n\}_{n=1}^\infty$ converges 
for any uniformly smooth Banach space $X \in \mathcal{P}_q$, 
any dictionary $\D$, any element $f \in X$ if and only if
\[
\sum_{n=1}^\infty t^p_n = \infty,
\]
where $p = q/(q-1)$.
\end{referencetheorem}

The next theorem gives the necessary and sufficient conditions for the 
convergence of the AWCGA (see Theorem~1 from~\cite{dereventsov_awcgausbs}).
\begin{referencetheorem}
The AWCGA with a weakness sequence $\{t_n\}_{n=1}^\infty$, 
a perturbation sequence $\{\delta_n\}_{n=0}^{\infty}$,
and an error sequence $\{\eta_n\}_{n=1}^{\infty}$
converges for any uniformly smooth Banach space $X \in \mathcal{P}_q$, 
any dictionary $\D$, and any element $f \in X$ if and only if
\[
\eta_0 = \sup_{n \geq 1} \eta_n < \infty,
\]
and if there exists a subsequence $\{n_k\}_{k=1}^\infty$ such that 
the following conditions hold:
\begin{align*}
&\sum_{k=1}^\infty t^p_{n_k+1} = \infty,
\\
&\delta_{n_k} = o(t^p_{n_k+1}),
\\
&\eta_{n_k} = o(t^p_{n_k+1}),
\end{align*}
where $p = q/(q-1)$.
\end{referencetheorem}

We will prove the following result that states the necessary and 
sufficient conditions for the convergence of the gAWCGA.
\begin{theorem}\label{theorem_gawcga_convergence_in_P_q}
The gAWCGA with a weakness sequence $\{(t_n, t'_n)\}_{n=1}^\infty$, 
a perturbation sequence $\{(\delta_n, \delta'_n)\}_{n=0}^{\infty}$,
and an error sequence $\{(\eta_n, \eta'_n)\}_{n=1}^{\infty}$
converges for any uniformly smooth Banach space $X \in \mathcal{P}_q$, 
any dictionary $\D$, and any element $f \in X$ if and only if 
\begin{equation}
\label{gawcga_necessity_in_P_q_eta_boundedness}
\eta_0 = \sup_{n \geq 1} \eta_n < \infty, \ \ 
\lim_{n\to\infty}\eta'_n = 0, 
\end{equation}
and if there exists a subsequence $\{n_k\}_{k=1}^{\infty}$ such that
the following conditions hold:
\begin{align}
\label{gawcga_convergence_in_P_q_sum_t_p=infty}
&\sum_{k=1}^\infty t_{n_k+1}^p = \infty,
\\
\label{gawcga_convergence_in_P_q_t'=o(t)}
&t'_{n_k+1} = o(t_{n_k+1}),
\\
\label{gawcga_convergence_in_P_q_delta=o(t_p)}
&\delta_{n_k} = o(t_{n_k+1}^p),
\\
\label{gawcga_convergence_in_P_q_delta'=o(t_p)}
&\delta'_{n_k} = o(t_{n_k+1}^p),
\\
\label{gawcga_convergence_in_P_q_eta=o(t_p)}
&\eta_{n_k} = o(t_{n_k+1}^p),
\\
\label{gawcga_convergence_in_P_q_eta'=o(t_p)}
&\eta'_{n_k} = o(t_{n_k+1}^p),
\end{align}
where $p = q/(q-1)$.
\end{theorem}

The following corollary states that if the weakness parameter 
$\{t_n\}_{n=1}^\infty$ is separated from zero (e.g. $t_n = t > 0$ for all $n$)
then the gAWCGA converges as long as $\eta'_n$ goes to zero and other
inaccuracy parameters go to zero along the same subsequence.
\begin{corollary}\label{corollary_gawcga_liminf_t_n>0}
Let $\liminf_{n\to\infty} t_n > 0$.
Then the gAWCGA with a weakness sequence $\{(t_n, t'_n)\}_{n=1}^\infty$, 
a perturbation sequence $\{(\delta_n, \delta'_n)\}_{n=0}^{\infty}$,
and a bounded error sequence $\{(\eta_n, \eta'_n)\}_{n=1}^{\infty}$
with $\lim_{n\to\infty} \eta'_n = 0$,
converges for any uniformly smooth Banach space $X \in \mathcal{P}_q$, 
any dictionary $\D$, and any element $f \in X$ if and only if
\[
\liminf_{n\to\infty} \br{t'_{n+1} + \delta_n + \delta'_n + \eta_n} = 0.
\]
\end{corollary}

The last two corollaries state that the conditions for the convergence 
of the gAWCGA are the same as for the WCGA
if inaccuracy sequences are from $\ell_1$ space.
\begin{corollary}\label{corollary_gawcga_l_1}
Let $\{t'_n\}_{n=1}^\infty \in \ell_1$, $\{\delta_n\}_{n=0}^\infty \in \ell_1$, 
$\{\delta'_n\}_{n=0}^\infty \in \ell_1$, $\{\eta_n\}_{n=1}^\infty \in \ell_1$, 
and $\{\eta'_n\}_{n=1}^\infty \in \ell_1$.
Then the gAWCGA with a weakness sequence $\{(t_n, t'_n)\}_{n=1}^\infty$, 
a perturbation sequence $\{(\delta_n, \delta'_n)\}_{n=0}^{\infty}$,
and a bounded error sequence $\{(\eta_n, \eta'_n)\}_{n=1}^{\infty}$
converges for any uniformly smooth Banach space $X$, 
any dictionary $\D$, and any element $f \in X$ if 
for any $0 < \theta \leq 1/2$
\[
\sum_{n=1}^\infty t_n \xi_n = \infty.
\]
\end{corollary}

\begin{corollary}\label{corollary_gawcga_l_1_P_q}
Let $\{t'_n\}_{n=1}^\infty \in \ell_1$, $\{\delta_n\}_{n=0}^\infty \in \ell_1$, 
$\{\delta'_n\}_{n=0}^\infty \in \ell_1$, $\{\eta_n\}_{n=1}^\infty \in \ell_1$, 
and $\{\eta'_n\}_{n=1}^\infty \in \ell_1$.
Then the gAWCGA with a weakness sequence $\{(t_n, t'_n)\}_{n=1}^\infty$, 
a perturbation sequence $\{(\delta_n, \delta'_n)\}_{n=0}^{\infty}$,
and a bounded error sequence $\{(\eta_n, \eta'_n)\}_{n=1}^{\infty}$
converges for any uniformly smooth Banach space $X \in \mathcal{P}_q$, 
any dictionary $\D$, and any element $f \in X$ if and only if
\[
\sum_{n=1}^\infty t^p_n = \infty.
\]
\end{corollary}

We note that in the last corollary the sequence $\{t'_n\}_{n=1}^\infty$ 
might be from $\ell_p$ space as well, however we consider it to be from 
$\ell_1$ for the simplicity of the formulation.
Corollaries~\ref{corollary_gawcga_l_1} and~\ref{corollary_gawcga_l_1_P_q} are 
obtained using Theorems~\ref{gawcga_convergence} and~\ref{theorem_gawcga_convergence_in_P_q}, 
and the following simple fact (see Lemma~2 from~\cite{dereventsov_awcgausbs}).
\begin{referencelemma}\label{lemma_sequences_l_1}
Let $\{a_n\}_{n=1}^\infty$ and $\{b_n\}_{n=1}^\infty$ be any 
such non-negative sequences that
\[
\sum_{n=1}^\infty a_n < \infty
\quad\text{and}\quad
\sum_{n=1}^\infty b_n = \infty.
\]
Then there exists a subsequence $\{n_k\}_{k=1}^\infty$ such that
\[
\sum_{k=1}^\infty b_{n_k} = \infty
\quad\text{and}\quad
a_{n_k} = o(b_{n_k}).
\]
\end{referencelemma}

\section{Proofs of Theorems~\ref{gawcga_convergence} 
and~\ref{theorem_gawcga_convergence_in_P_q}}
\label{proofs}

We will need several technical results.
The following is Lemma~2.2 from~\cite{temlyakov_gabs}.
\begin{referencelemma}\label{sup(D)=supA1(D)}
For any bounded linear functional $F$ and any
dictionary $\D$
\[
\sup_{g\in\D} |F(g)| = \sup_{g\in A_1(\D)} |F(g)|.
\]
\end{referencelemma}

We use the following lemmas from~\cite{temlyakov_gtabsa} rewritten for the gAWCGA.
\begin{lemma}\label{F_n(phi_n)}
Let $X$ be a Banach space with the modulus of smoothness $\rho(u)$. 
Then for any $\phi \in \Phi_n$
\[
|F_n(\phi)| \leq \beta_n(\phi) 
:= \inf_{\lambda>0} \frac{1}{\lambda}
\br{\delta_n + \eta_n 
+ \frac{\delta'_n + \eta'_n}{\n{f_n}} 
+ 2\rho(\lambda\n{\phi})}.
\]
\end{lemma}

\begin{proof}
Take any $\phi$ from $\Phi_n$. By the definition of the modulus of 
smoothness~\eqref{modulusofsmoothness} for any $\lambda > 0$
\[
\n{f_n - \lambda\phi} + \n{f_n + \lambda\phi} 
\leq 2\n{f_n} \br{1 + \rho\br{\frac{\lambda\n{\phi}}{\n{f_n}}}}.
\]
Assume that $F_n(\phi) \geq 0$ (case $F_n(\phi) < 0$ is handled similarly). 
Then, using~\eqref{gawcga_functional}, we obtain
\[
\n{f_n + \lambda\phi} \geq F_n\br{f_n + \lambda\phi}
\geq (1-\delta_n)\n{f_n} - \delta'_n + \lambda F_n(\phi),
\]
thus
\[
\n{f_n - \lambda\phi} \leq 
\n{f_n}\br{1 + \delta_n + 2\rho\br{\frac{\lambda\n{\phi}}{\n{f_n}}}} 
+ \delta'_n - \lambda F_n(\phi).
\]
On the other hand, by~\eqref{gawcga_approximation}
\[
\n{f_n - \lambda\phi} \geq E_n 
\geq (1+\eta_n)^{-1} \br{\n{f_n} - \eta'_n}
\geq (1-\eta_n) \n{f_n} - \eta'_n.
\]
Therefore
\[
\lambda F_n(\phi) 
\leq \n{f_n}\br{\delta_n + \eta_n + 2\rho\br{\frac{\lambda\n{\phi}}{\n{f_n}}}}
+ \delta'_n + \eta'_n
\]
and, since the inequality holds for any $\lambda > 0$,
\[
F_n(\phi) 
\leq \inf_{\lambda>0} \frac{1}{\lambda}
\br{\delta_n + \eta_n + \frac{\delta'_n + \eta'_n}{\n{f_n}} + 2\rho(\lambda\n{\phi})}
= \beta_n(\phi).
\]
\end{proof}

\begin{lemma}\label{F_n(phi_n+1)}
Let $X$ be a Banach space with the modulus of
smoothness $\rho(u)$. 
Take a number $\epsilon \geq 0$ and two elements 
$f$ and $h$ from $X$ such that $\n{f-h} \leq \epsilon$
and $h/A \in A_1(\D)$ with some number $A = A(\epsilon) > 0$.
Then
\begin{equation*}
|F_n(\phi_{n+1})|
\geq t_{n+1} A^{-1} \br{(1 - \delta_n) \n{f_n} - \delta'_n 
- \beta_n(G_n) - \epsilon} - t'_{n+1}.
\end{equation*}
\end{lemma}

\begin{proof}
Condition~\eqref{gawcga_approximant} and Lemma~\ref{sup(D)=supA1(D)} provide
\[
|F_n(\phi_{n+1})| 
\geq t_{n+1} \sup_{g\in\D} |F_n(g)| - t'_{n+1}  
= t_{n+1} \sup_{g\in A_1(\D)} |F_n(g)| - t'_{n+1}.
\]
Taking $g = h/A \in A_1(\D)$ we obtain
\begin{align*}
\sup_{g\in A_1(\D)} |F_n(g)|
&\geq A^{-1} |F_n(h)|
\geq A^{-1} \br{|F_n(f)| - \epsilon}
\\
&\geq A^{-1} \br{|F_n(f_n)| - |F_n(G_n)| - \epsilon}.
\end{align*}
Hence condition~\eqref{gawcga_functional} and Lemma~\ref{F_n(phi_n)} provide
\[
|F_n(\phi_{n+1})|
\geq t_{n+1} A^{-1} \br{(1 - \delta_n) \n{f_n} - \delta'_n 
- \beta_n(G_n) - \epsilon} - t'_{n+1}.
\]
\end{proof}

\begin{lemma}\label{E_m<E_n}
Let $X$ be a Banach space with the modulus of smoothness $\rho(u)$. 
Take a number $\epsilon \geq 0$ and two elements 
$f$ and $h$ from $X$ such that $\n{f-h} \leq \epsilon$
and $h/A \in A_1(\D)$ with some number $A = A(\epsilon) > 0$.
Then for any $m > n$
\begin{multline*}
E_m \leq \inf_{\mu\geq0} \n{f_n} \Bigg[
1 + \delta_n + \frac{\delta'_n}{\n{f_n}} + 2\rho\br{\frac{\mu}{\n{f_n}}}
\\
- \frac{\mu t_{n+1}}{A\n{f_n}} \Big((1 - \delta_n) \n{f_n} 
- \delta'_n - \beta_n(G_n) - \epsilon\Big) \Bigg] + \mu t'_{n+1}.
\end{multline*}
\end{lemma}

\begin{proof}
By the definition of the modulus of 
smoothness~\eqref{modulusofsmoothness} for any $\mu \geq 0$
\[
\n{f_n - \mu\phi_{n+1}} + \n{f_n + \mu\phi_{n+1}} 
\leq 2\n{f_n} \br{1 + \rho\br{\frac{\mu}{\n{f_n}}}}.
\]
Assume that $F_n(\phi_{n+1}) \geq 0$ 
(case $F_n(\phi_{n+1}) < 0$ is handled similarly). 
Then, using~\eqref{gawcga_functional} and Lemma~\ref{F_n(phi_n+1)}, we get
\begin{align*}
\n{f_n + \mu\phi_{n+1}} 
&\geq F_n\br{f_n + \mu\phi_{n+1}}
\geq (1 - \delta_n)\n{f_n} - \delta'_n + \mu |F_n(\phi_{n+1})|
\\
&\geq \begin{aligned}[t]
(1 &- \delta_n)\n{f_n} - \delta'_n 
\\
&+ \mu t_{n+1} A^{-1} \br{(1 - \delta_n) \n{f_n} - \delta'_n 
- \beta_n(G_n) - \epsilon} - \mu t'_{n+1}.
\end{aligned}
\end{align*}
Thus
\begin{multline*}
\n{f_n - \mu\phi_{n+1}} \leq \n{f_n} 
\br{1 + \delta_n + \frac{\delta'_n}{\n{f_n}} + 2\rho\br{\frac{\mu}{\n{f_n}}}}
\\
- \mu t_{n+1} A^{-1} \Big((1 - \delta_n) \n{f_n} 
- \delta'_n - \beta_n(G_n) - \epsilon\Big) + \mu t'_{n+1}.
\end{multline*}
On the other hand, since 
$E_m \leq E_{n+1} \leq \n{f_n - \mu\phi_{n+1}}$ 
for any $\mu \geq 0$,
\begin{multline*}
E_m \leq \n{f_n} \Bigg[
1 + \delta_n + \frac{\delta'_n}{\n{f_n}} + 2\rho\br{\frac{\mu}{\n{f_n}}}
\\
- \frac{\mu t_{n+1}}{A\n{f_n}} \Big((1 - \delta_n) \n{f_n} 
- \delta'_n - \beta_n(G_n) - \epsilon\Big)\Bigg] + \mu t'_{n+1}.
\end{multline*}
Taking an infimum over all $\mu \geq 0$ completes the proof.
\end{proof}

We are now ready to prove Theorem~\ref{gawcga_convergence}.

\begin{proof}[Proof of Theorem~\ref{gawcga_convergence}]
Assume that for some element $f \in X$ the gAWCGA does not converge.
Note that then the monotone sequence $\{E_n\}_{n=1}^\infty$ does not converge 
to $0$ since otherwise conditions~\eqref{gawcga_eta_boundedness} would imply 
\[
\lim_{n\to\infty} \n{f_n} 
\leq \lim_{n\to\infty} \br{(1 + \eta_0) E_n + \eta'_n} = 0.
\]
Thus there exists a number $\alpha > 0$ such that for any $n \geq 1$
\begin{equation}\label{gawcga_assumption_E_n>alpha}
\n{f_n} \geq E_n \geq \alpha.
\end{equation}
Denote $C_f = (2 + \eta_0)\n{f} + \eta'_0 < \infty$, 
where $\eta_0 = \sup_{n \geq 1} \eta_n$ and $\eta'_0 = \sup_{n \geq 1} \eta'_n$.
Then inequality~\eqref{gawcga_approximation} gives for any $n \geq 1$
\begin{equation}
\label{gawcga_convergence_f_n_G_n_estimate}
\begin{aligned}
&\n{f_n} \leq (1 + \eta_0)\n{f} + \eta'_0 \leq C_f,
\\
&\n{G_n} \leq \n{f_n} + \n{f} \leq C_f.
\end{aligned}
\end{equation}
Let $\{n_k\}_{k=1}^\infty$ be a subsequence for which 
conditions of the theorem hold. Then 
\begin{align*}
\beta_{n_k}(G_{n_k}) 
&= \inf_{\lambda>0} \frac{1}{\lambda}
\br{\delta_{n_k} + \eta_{n_k} + \frac{\delta'_{n_k} + \eta'_{n_k}}{\n{f_{n_k}}} 
+ 2\rho(\lambda\n{G_{n_k}})}
\\
&\leq \inf_{\lambda>0} \frac{1}{\lambda}
\br{\delta_{n_k} + \eta_{n_k} + \frac{\delta'_{n_k} + \eta'_{n_k}}{\alpha} 
+ 2\rho\br{\lambda C_f}}
\end{align*}
and, due to conditions~\eqref{gawcga_t'=o(t)}--\eqref{gawcga_eta'=o(t_xi)} 
and the inequality $0 \leq \theta t_n \xi_n \leq 1$, 
there exists a number $K \geq 1$ such that for any $k \geq K$ the 
following estimates hold with $\theta =\frac{\alpha^2}{24 A C_f}$:
\begin{align}
\label{gawcga_convergence_K1}
& \br{\frac{1}{2} - \delta_{n_k}} \alpha - \delta'_{n_k} - \beta_{n_k}(G_{n_k}) 
\geq \frac{\alpha}{4},
\\
\label{gawcga_convergence_K2}
& \delta_{n_k} + \frac{\delta'_{n_k}}{\alpha} \leq \theta \xi_{n_k+1} t_{n_k+1},
\\
\label{gawcga_convergence_K3}
& (1 + \eta_{n_k})(1 - 3\theta \xi_{n_k+1} t_{n_k+1}) 
\leq 1 - 2\theta \xi_{n_k+1} t_{n_k+1},
\\
\label{gawcga_convergence_K4}
& \eta'_{n_k} + \alpha \xi_{n_k+1} t'_{n_k+1} \leq \alpha\theta \xi_{n_k+1} t_{n_k+1}.
\end{align}
Take $\epsilon = \alpha/2$ and find an element 
$h \in X$ such that $\n{f-h} \leq \epsilon$ 
and ${h/A \in A_1(\D)}$ for some $A > 0$.
Then Lemma~\ref{E_m<E_n}, assumption~\eqref{gawcga_assumption_E_n>alpha}, 
and estimates~\eqref{gawcga_convergence_f_n_G_n_estimate} 
and~\eqref{gawcga_convergence_K1} provide for any $k \geq K$
\begin{align*}
E_{n_{k+1}} 
&\leq \inf_{\mu\geq0} \n{f_{n_k}} \Bigg[
1 + \delta_{n_k} + \frac{\delta'_{n_k}}{\alpha} + 2\rho\br{\frac{\mu}{\alpha}}
\\
&\qquad - \frac{\mu t_{{n_k}+1}}{A C_f} 
\br{\br{\frac{1}{2} - \delta_{n_k}} \alpha 
- \delta'_{n_k} - \beta_{n_k}(G_{n_k})} \Bigg] + \mu t'_{{n_k}+1}
\\
&\leq \inf_{\mu\geq0} \n{f_{n_k}} \Bigg[
1 + \delta_{n_k} + \frac{\delta'_{n_k}}{\alpha} + 2\rho\br{\frac{\mu}{\alpha}}
- \frac{\alpha \mu t_{{n_k}+1}}{4A C_f} \Bigg] + \mu t'_{{n_k}+1}.
\end{align*}
By taking $\mu = \alpha \xi_{n_k+1}$, and using 
estimates~\eqref{gawcga_convergence_K2}--\eqref{gawcga_convergence_K4}, 
and condition~\eqref{gawcga_approximation} we obtain
\begin{align}
\nonumber
E_{n_{k+1}} &\leq 
\n{f_{n_k}} \br{1 + \delta_{n_k} + \frac{\delta'_{n_k}}{\alpha} 
- 4\theta \xi_{n_k+1} t_{n_k+1}} + \alpha \xi_{n_k+1} t'_{n_k+1}
\\
\nonumber
&\leq \n{f_{n_k}} \br{1 - 3\theta \xi_{n_k+1} t_{{n_k}+1}} 
+ \alpha \xi_{n_k+1} t'_{n_k+1}
\\
\nonumber
&\leq E_{n_k} \br{1 - 2\theta \xi_{n_k+1} t_{n_k+1}} + \eta'_{n_k} 
+ \alpha \xi_{n_k+1}t'_{n_k+1}
\\
\label{gawcga_convergence_recursive_estimate_E_n}
&\leq E_{n_k} \br{1 - \theta \xi_{n_k+1} t_{n_k+1}}.
\end{align}
Note that condition~\eqref{gawcga_sum_t_xi=infty} implies that the infinite product 
$\prod_{k=1}^{\infty} \br{1 - \theta \xi_{n_k+1} t_{n_k+1}}$ diverges to $0$.
Then, recursively applying estimate~\eqref{gawcga_convergence_recursive_estimate_E_n},
we obtain for sufficiently big $N \geq K$
\begin{align*}
E_{n_{N+1}} 
&\leq E_{n_{K}} \prod_{k=K}^{N} \br{1 - \theta \xi_{n_k+1} t_{n_k+1}} 
\\
&\leq \n{f} \prod_{k=K}^{N} \br{1 - \theta \xi_{n_k+1} t_{n_k+1}} 
\\
&< \alpha,
\end{align*}
which contradicts assumption~\eqref{gawcga_assumption_E_n>alpha}. 
Therefore
\[
\lim_{n\to\infty} E_n = 0,
\]
i.e. the gAWCGA of $f$ converges to $f$.
\end{proof}

We will use the following simple lemma to prove 
Theorem~\ref{theorem_gawcga_convergence_in_P_q}.
\begin{lemma}\label{(a+b^q)^1/q<=a+b}
Let $q > 1$, $a \geq 0$ and $b \geq 1$.
Then
\[
\br{a + b^q}^{1/q} \leq a + b.
\]
\end{lemma}
\begin{proof}
Due to the convexity of $(1 + x)^q$ we have for any $x \geq 0$
\[
(1 + x)^q \geq 1 + qx.
\]
Then by taking $x = a/b$ we get
\[
(a + b)^q = b^q (1 + x)^q 
\geq b^q (1 + qx)
= b^q + aqb^{q-1}
\geq a + b^q.
\]
\end{proof}

\begin{proof}[Proof of Theorem~\ref{theorem_gawcga_convergence_in_P_q}]
We start with the proof of the sufficiency. 
Assume that 
conditions~\eqref{gawcga_convergence_in_P_q_sum_t_p=infty}--\eqref{gawcga_convergence_in_P_q_eta'=o(t_p)}
hold for some subsequence $\{n_k\}_{k=1}^\infty$.
Choose any number $0 < \theta \leq 1/2$ and find the corresponding 
sequence $\{\xi_n\}_{n=1}^\infty$. Then using the definition 
$\rho(\xi_n) = \theta t_n \xi_n$ and the estimate 
$\rho(u) \leq \gamma u^q$, we derive
\[
\xi_n \geq \br{\frac{\theta}{\gamma} t_n}^{p-1}.
\]
Thus for any $n \geq 1$
\[
t_n^p \leq \br{\frac{\gamma}{\theta}}^{p-1} t_n \xi_n,
\]
and conditions~\eqref{gawcga_convergence_in_P_q_sum_t_p=infty}--\eqref{gawcga_convergence_in_P_q_eta'=o(t_p)}
imply that conditions~\eqref{gawcga_sum_t_xi=infty}--\eqref{gawcga_eta'=o(t_xi)} 
hold for the subsequence $\{n_k\}_{k=1}^\infty$ and any $0 < \theta \leq 1/2$. 
Therefore Theorem~\ref{gawcga_convergence}
guarantees convergence of the gAWCGA for any dictionary $\D$ 
and any element $f \in X$.

\smallskip
We prove the necessity of the stated conditions by giving a counterexample.
Namely, we assume that at least one of 
conditions~\eqref{gawcga_necessity_in_P_q_eta_boundedness}--\eqref{gawcga_convergence_in_P_q_eta'=o(t_p)}
fails, and give an example of such a Banach space $X \in \mathcal{P}_q$,
a dictionary $\D$, and an element $f \in \D$ that the gAWCGA of $f$ diverges.

Let $X = \ell_q \in \mathcal{P}_q$ and $\D = \{\pm e_n\}_{n=0}^\infty$, 
where $\{e_n\}_{n=0}^\infty$ is the canonical basis in $\ell_q$.

Assume that condition~\eqref{gawcga_necessity_in_P_q_eta_boundedness} fails, 
i.e. that there exist a subsequence $\{n_k\}_{k=1}^\infty$ and 
a number $\alpha > 0$ such that for any $k \geq 1$
\[
\eta_{n_k} \geq \alpha k
\ \ \text{or}\ \ 
\eta'_{n_k} \geq \alpha.
\]
Then take a positive non-increasing sequence $\{a_j\}_{j=1}^\infty \in \ell_q$ 
such that 
\[
a_1 \geq \alpha
\ \ \text{and}\ \ 
\br{\sum_{j=n_k+1}^\infty a_j^q}^{1/q} \geq k^{-1}
\]
for any $k \geq 1$.
Denote $f = \sum_{j=1}^\infty a_j e_j \in \ell_q$ and
consider the following realization of the gAWCGA of $f$:
\\
\indent For $n \not\in \{n_k\}_{k=1}^\infty$ choose $F_{n-1}$ to be the 
norming functional for $f_{n-1}$, $\phi_n = e_n$,  
and $G_n = \sum_{j=1}^n a_j e_j$.
\\
\indent For $n \in \{n_k\}_{k=1}^\infty$ choose $F_{n-1}$ to be the 
norming functional for $f_{n-1}$, $\phi_n = e_n$
and $G_n = \alpha e_1 + \sum_{j=1}^n a_j e_j$, 
which is possible since
\[
\n{f_{n_k}}_q = \br{\alpha^q + \sum_{j=n_k+1}^\infty a_j^q}^{1/q}
\leq \alpha + E_{n_k},
\]
and either
\[
\n{f_{n_k}}_q \leq E_{n_k} + \eta'_{n_k}
\ \ \text{or}\ \ 
\n{f_{n_k}}_q \leq (1 + \alpha k) E_{n_k}
\leq (1 + \eta_{n_k}) E_{n_k}.
\]
Then for any $k \geq 1$ norm of the remainder 
$\n{f_{n_k}}_q \geq \alpha$, hence $\n{f_n}_q \not\to 0$
and the gAWCGA of $f$ diverges.

Assume now that 
conditions~\eqref{gawcga_convergence_in_P_q_sum_t_p=infty}--\eqref{gawcga_convergence_in_P_q_eta'=o(t_p)} 
do not hold, i.e. for any subsequence $\{n_k\}_{k=1}^\infty$ at least one 
of the following statements fails:
\begin{align*}
&\sum_{k=1}^\infty t_{n_k+1}^p = \infty,
\\
&t'_{n_k+1} = o(t_{n_k+1})
\\
&\delta_{n_k} = o(t_{n_k+1}^p),
\\
&\delta'_{n_k} = o(t_{n_k+1}^p),
\\
&\eta_{n_k} = o(t_{n_k+1}^p),
\\
&\eta'_{n_k} = o(t_{n_k+1}^p).
\end{align*}
For a number $\alpha > 0$ define sets
\[
\Lambda_1 = \{n > 1 : 
\delta_{n-1} + \delta'_{n-1} \geq \alpha t^p_n 
\text{ or } 
\eta_{n-1} + \eta'_{n-1} \geq \alpha t^p_n
\text{ or }
t'_n \geq \alpha^{1/p} t_n\}
\]
and $\Lambda_2 = \mathbb{N} \setminus \Lambda_1$.
We claim that there exists an $\alpha > 0$ such that
\begin{equation}
\label{gawcga_convergence_necessity_alpha_claim}
\sum_{j \in \Lambda_2} t^p_j < \infty.
\end{equation}
Indeed, if $\sum_{j \in \Lambda_2} t^p_j = \infty$ for any $\alpha > 0$ 
then for every $k \geq 1$ consider $\alpha(k) = 1/k$, 
and choose a sequence of disjoint finite sets $\{\Gamma_k\}_{k=1}^\infty$, 
where $\Gamma_k \subset \Lambda_2(k)$ is such that 
$\sum_{j \in \Gamma_k} t^p_j \geq 1$.
Hence by considering the union $\cup_{k=1}^{\infty} (\Gamma_k + \{-1\})$
(where $+$ denotes the Minkowski addition),
we receive the subsequence for which 
conditions~\eqref{gawcga_convergence_in_P_q_sum_t_p=infty}--\eqref{gawcga_convergence_in_P_q_eta'=o(t_p)}
hold, which contradicts the aforementioned assumption.
Fix an $\alpha > 0$ for which claim~\eqref{gawcga_convergence_necessity_alpha_claim}
holds, and find corresponding sets $\Lambda_1$ and $\Lambda_2$.

If $|\Lambda_1| < \infty$ then $\sum_{j=1}^\infty t^p_j < \infty$.
Take $f = e_0 + \sum_{j=1}^\infty t^{p/q}_j e_j$ and 
consider the following realization of the gAWCGA of $f$:
\\
\indent For each $n \geq 1$ choose $F_{n-1}$ to be the 
norming functional for $f_{n-1}$, $\phi_n = e_n$, 
and $G_n = \sum_{j=1}^n t^{p/q}_j e_j$.
Then for any $n \geq 1$ norm of the remainder 
$\n{f_n}_q \geq 1$, hence the gAWCGA of $f$ diverges.

Consider the case $|\Lambda_1| = \infty$.
Take any such non-negative sequence $\{a_j\}_{j\in\Lambda_1}$
that $a_j \leq 1$ for any $j \geq 1$, 
$\sum_{j\in\Lambda_1} a_j^q \geq 1/\alpha$
and $\sum_{j\in\Lambda_1} a_j^p < \infty$.
Denote 
\[
f = \alpha^{1/q} \beta \br{\sum_{j \in \Lambda_1} a_j e_j 
+ \sum_{j \in \Lambda_2} t^{p/q}_j e_j},
\]
where 
\[
\beta = \br{\eta_0 + \eta'_0 + 
\alpha\br{\sum_{j \in \Lambda_1} a_j^q 
+ \sum_{j \in \Lambda_2} t^p_j}}^{-1/q}
\leq 1.
\]
We claim that for some realization of the gAWCGA of $f$ the 
indices from $\Lambda_1$ will not be chosen.
Namely, we show that there exists such a realization that
for any $n \geq 1$ the set $\Gamma_n$ of indices of $e_j$ chosen 
on the first $n$ steps of the algorithm and the $n$-th remainder 
$f_n$ satisfy the following relations:
\begin{equation}
\label{awcga_convergence_awcga_claim}
\begin{aligned}
&\Gamma_n \cap \Lambda_1 = \varnothing,
\\
&f_n = \beta(\eta_n + \eta'_n)^{1/q} e_1 
+ \alpha^{1/q} \beta \br{\sum_{j \in \Lambda_1} a_j e_j 
+ \sum_{j \in \Lambda_2^{(n)}} t^{p/q}_j e_j},
\end{aligned}
\end{equation}
where $\Lambda_2^{(n)} = \Lambda_2 \setminus \Gamma_n$.
Consider the following realization of the gAWCGA of $f$:
\\
\indent For $n = 1$ choose 
\[
F_0(x) = F_f(x) 
= \frac{\sum_{j \in \Lambda_1} a_j^{q/p} x_j 
+ \sum_{j \in \Lambda_2} t_j x_j}
{(\alpha\beta^q)^{-1/p} \n{f}_q^{q/p}}.
\]
Then, since $a_j \leq 1$, we get
\begin{align*}
F_0(e_0) &= 0,
\\
F_0(e_j) &\leq (\alpha\beta^q)^{1/p} \n{f}_q^{-q/p},
\text{ for any } j \in \Lambda_1,
\\
F_0(e_j) &= t_j (\alpha\beta^q)^{1/p} \n{f}_q^{-q/p}
\text{ for any } j \in \Lambda_2,
\end{align*}
and choosing $\phi_1 = e_1$ satisfies~\eqref{gawcga_approximant} 
since $1 \in \Lambda_2$.
Thus $\Gamma_1 = \{1\}$, and taking 
\[
f_1 = \beta(\eta_1 + \eta'_1)^{1/q} e_1 
+ \alpha^{1/q} \beta\br{\sum_{j \in \Lambda_1} a_j e_j
+ \sum_{j \in \Lambda_2^{(1)}} t^{p/q}_j e_j}
\]
satisfies~\eqref{gawcga_approximation} since the estimate
\[
\beta \leq E_1 = \alpha^{1/q} \beta \br{\sum_{j \in \Lambda_1} a_j^q 
+ \sum_{j \in \Lambda_2^{(1)}} t^p_j}^{1/q} \leq 1
\] 
and Lemma~\ref{(a+b^q)^1/q<=a+b} provide
\begin{align*}
\n{f_1}_q &= \beta\br{\eta_1 + \eta'_1 
+ \alpha \br{\sum_{j \in \Lambda_1} a^q_j
+ \sum_{j \in \Lambda_2^{(1)}} t^p_j}}^{1/q}
\\
&\leq \beta(\eta_1 + \eta'_1) + E_1
\\
&\leq (1 + \eta_1) E_1 + \eta'_1.
\end{align*}
Hence for $n = 1$ claim~\eqref{awcga_convergence_awcga_claim} holds.
\\
\indent For $n \geq 1$, provided 
\[
f_n = \beta(\eta_n + \eta'_n)^{1/q} e_1 
+ \alpha^{1/q} \beta \br{\sum\limits_{j \in \Lambda_1} a_j e_j 
+ \sum_{j \in \Lambda_2^{(n)}} t^{p/q}_j e_j},
\]
taking 
\[
F_n(x) = 
\frac{(\delta_n + \delta'_n)^{1/p} x_0 
+ (\eta_n + \eta'_n)^{1/p} x_1 
+ \alpha^{1/p} \br{\sum\limits_{j \in \Lambda_1} a_j^{q/p} x_j 
+ \sum\limits_{j \in \Lambda_2^{(n)}} t_j x_j}}
{\br{\beta^{-q} (1 + \delta_n + \delta'_n) \n{f_n}_q^q}^{1/p}}
\]
satisfies~\eqref{gawcga_functional} since the estimate
\[
\beta \leq \n{f_n}_q 
= \beta\br{\eta_n + \eta'_n 
+ \alpha\br{\sum_{j\in\Lambda_1} a_j^q 
+ \sum_{j\in\Lambda_2^{(n)}} t_j^p}}^{1/q}
\leq 1
\]
and H\"older's inequality provide
\[
|F_n(x)| 
\leq \frac{\br{\delta_n + \delta'_n + \beta^{-q}\n{f_n}_q^q}^{1/p} 
\br{\sum\limits_{j=0}^\infty x_j^q}^{1/q}}
{\br{\beta^{-q} (1 + \delta_n + \delta'_n) \n{f_n}_q^q}^{1/p}}
\leq \n{x}_q
\]
and
\[
F_n(f_n) = \frac{\n{f_n}_q^q}
{(1 + \delta_n + \delta'_n)^{1/p} \n{f_n}_q^{q/p}}
\geq (1 - \delta_n) \n{f_n}_q - \delta'_n,
\]
where the last inequality holds since $\n{f_n}_q \leq 1$ and 
\begin{equation*}
\begin{aligned}
(1 + \delta_n &+ \delta'_n)^{1/p} 
((1 - \delta_n) \n{f_n}_q - \delta'_n)
\\
&\begin{aligned}
\leq (1 + \delta_n &+ \delta'_n)^{1/p} 
(1 - \delta_n - \delta'_n) \n{f_n}_q
\\
&\begin{aligned}
= (1 - (\delta_n &+ \delta'_n)^2)^{1/p} 
(1 - \delta_n - \delta'_n)^{1/q} \n{f_n}_q
\\
&\leq \n{f_n}_q.
\end{aligned}
\end{aligned}
\end{aligned}
\end{equation*}
Hence such choice of a functional is admissible.
Let $A_n = \br{\beta^{-q} (1 + \delta_n + \delta'_n) \n{f_n}_q^q}^{-1/p}$.
Then, since $a_j \leq 1$, we get
\begin{align*}
F_n(e_0) &= (\delta_n + \delta'_n)^{1/p} A_n,
\\
F_n(e_1) &= (\eta_n + \eta'_n)^{1/p} A_n,
\\
F_n(e_j) &\leq \alpha^{1/p} A_n
\text{ for any } j \in \Lambda_1,
\\
F_n(e_j) &= t_j \alpha^{1/p} A_n
\text{ for any } j \in \Lambda_2^{(n)},
\\
F_n(e_j) &= 0
\text{ for any } j \in \Gamma_n \setminus \{0, 1\}.
\end{align*}
If $n+1 \in \Lambda_2$ we choose $\phi_{n+1} = e_{n+1}$.
Otherwise $n+1 \in \Lambda_1$, and by definition of the set
at least one of the following inequalities holds:
\begin{align*}
&F_n(e_0) \geq t_{n+1} \alpha^{1/p} A_n
\geq t_{n+1} \alpha^{1/p} A_n - t'_{n+1},
\\
&F_n(e_1) \geq t_{n+1} \alpha^{1/p} A_n
\geq t_{n+1} \alpha^{1/p} A_n - t'_{n+1},
\\
&t_{n+1} \sup_{g \in \D} F_n(g) - t'_{n+1}
\leq t_{n+1} \alpha^{1/p} A_n - \alpha^{1/p} t_{n+1}
\leq 0.
\end{align*}
Then we choose $\phi_{n+1} = e_0$ or $\phi_{n+1} = e_1$.
In either case $\Gamma_{n+1} \cap \Lambda_1 = \varnothing$ and taking 
\[
f_{n+1} = \beta(\eta_{n+1} + \eta'_{n+1})^{1/q} e_1 
+ \alpha^{1/q} \beta\br{\sum\limits_{j \in \Lambda_1} a_j e_j 
+ \sum\limits_{j \in \Lambda_2^{(n+1)}} t^{p/q}_j e_j}
\]
satisfies~\eqref{gawcga_approximation} since the estimate
\[
\beta \leq E_{n+1} = \alpha^{1/q} \beta \br{\sum_{j \in \Lambda_1} a_j^q 
+ \sum_{j \in \Lambda_2^{(n+1)}} t^p_j}^{1/q} \leq 1
\] 
and Lemma~\ref{(a+b^q)^1/q<=a+b} provide
\begin{align*}
\n{f_{n+1}}_q &= \beta\br{\eta_{n+1} + \eta'_{n+1} 
+ \alpha \br{\sum_{j \in \Lambda_1} a^q_j
+ \sum_{j \in \Lambda_2^{(n+1)}} t^p_j}}^{1/q}
\\
&\leq \beta(\eta_{n+1} + \eta'_{n+1}) + E_{n+1}
\\
&\leq (1 + \eta_{n+1}) E_{n+1} + \eta'_{n+1}.
\end{align*}
Hence claim~\eqref{awcga_convergence_awcga_claim} holds for any $n \geq 1$.
Thus $\n{f_n} \geq \beta \not\to 0$ and the gAWCGA of $f$ diverges.
\end{proof}

\section{Non-sufficiency of smoothness of a space for the convergence of the WCGA}
\label{smoothness_non-sufficiency_wcga}

In this section we demonstrate that smoothness of a space is not a sufficient 
condition for the convergence of the WCGA.
Specifically, we will construct an example of a smooth Banach space $X$, 
a dictionary $\D$, and an element ${f \in X}$ such that the WCGA of $f$
with any weakness sequence $\{t_n\}_{n=1}^\infty$ diverges.
The space we construct was used in~\cite{donahue_darken_gurvits_sontag_rcanhs} and, 
in a special case, in~\cite{livshitz_cgabs}.

In order to obtain the desired space we re-norm $\ell_1$ --- the space of sequences, 
whose series are absolutely convergent.
Take a non-increasing sequence of numbers $\{p_n\}_{n=1}^\infty$ 
with $p_n > 1$ for any $n \geq 1$, and
\begin{equation}\label{p_n_condition_smoothness}
\sum_{n=1}^{\infty} \br{1 - \frac{1}{p_n}} < \infty.
\end{equation}
Let $\{e_n\}_{n=1}^\infty$ be the canonical basis in $\ell_1$.
Consider a sequence of non-linear functionals $\{\vartheta_n\}_{n=0}^\infty$ defined 
as follows: for any $x = \sum_{n=1}^\infty x_n e_n \in \ell_1$
\begin{align*}
\vartheta_0(x) &= 0,
\text{ and for any } n \geq 1
\\
\vartheta_n(x) &= \br{\vartheta_{n-1}^{p_n}(x) + |x_n|^{p_n}}^{1/p_n}.
\end{align*}
In particular,
\begin{align*}
\vartheta_1(x) &= |x_1|,
\\
\vartheta_2(x) &= \br{|x_1|^{p_2} + |x_2|^{p_2}}^{1/p_2},
\\
\vartheta_3(x) &= \br{\br{|x_1|^{p_2} + |x_2|^{p_2}}^{p_3/p_2} + |x_3|^{p_3}}^{1/p_3}.
\end{align*}
We claim that $\vartheta_n$ is a norm on $\ell_1^n$. Indeed, for any $x \in \ell_1^n$
\begin{gather*}
\vartheta_n(x) = 0
\text{ if and only if } x = 0, 
\\
\vartheta_n(\lambda x) = |\lambda| \vartheta_n(x)
\text{ for any } \lambda \in \mathbb{R}.
\end{gather*}
We prove triangle inequality for $\vartheta_n(\cdot)$ using induction by $n$.
The base case $n = 1$ is obvious. Then, using Minkowski's inequality, we
obtain for any $n > 1$ and any $x, y \in \ell_1^n$
\begin{align*}
\vartheta_n(x+y) &= \br{\vartheta_{n-1}^{p_n}(x+y) + |x_n + y_n|^{p_n}}^{1/p_n}
\\
&\leq \br{\br{\vartheta_{n-1}(x) + \vartheta_{n-1}(y)}^{p_n} + \br{|x_n| + |y_n|}^{p_n}}^{1/p_n}
\\
&\leq \br{\vartheta_{n-1}^{p_n}(x) + |x_n|^{p_n}}^{1/p_n} 
+ \br{\vartheta_{n-1}^{p_n}(y) + |y_n|^{p_n}}^{1/p_n}
\\
&= \vartheta_n(x) + \vartheta_n(y).
\end{align*}
Define the space $X$ as 
\begin{equation*}
X = \{x \in \ell_1 : \lim_{n\to\infty} \vartheta_n(x) < \infty\},
\end{equation*}
and the norm $\n{\cdot}_X$ on X as
\begin{equation*}
\n{x}_X = \lim_{n\to\infty} \vartheta_n(x).
\end{equation*}
Note that for any $x \in \ell_1$ the sequence $\{\vartheta_n(x)\}_{n=0}^\infty$ 
is non-decreasing, and, therefore, the limit always exists.
Moreover, for any $n \geq 1$
\begin{equation*}
\vartheta_n(x) \leq \vartheta_{n-1}(x) + |x_n| \leq \sum_{k=1}^n |x_k|,
\end{equation*}
and, by H\"older's inequality,
\begin{align*}
\sum_{k=1}^n |x_k| &\leq 2^{1-\frac{1}{p_2}}\vartheta_2(x) + \sum_{k=3}^n |x_k|
\leq 2^{1-\frac{1}{p_2}} \br{\vartheta_2(x) + \sum_{k=3}^n |x_k|}
\\
&\leq 2^{1-\frac{1}{p_2}}\br{2^{1-\frac{1}{p_3}}\vartheta_3(x) + \sum_{k=4}^n |x_k|}
\leq 2^{\sum_{k=2}^3 \br{1-\frac{1}{p_k}}} \br{\vartheta_3(x) + \sum_{k=4}^n |x_k|}
\\
&\cdots
\\
&\leq 2^{\sum_{k=2}^{n-1} \br{1-\frac{1}{p_k}}} \br{\vartheta_{n-1}(x) + |x_n|}
\leq 2^{\sum_{k=2}^n \br{1-\frac{1}{p_k}}} \vartheta_n(x).
\end{align*}
Therefore, by taking the limit by $n \to \infty$, we obtain for any $x \in X$
\begin{equation}\label{1_X_norm_equivalence}
\rho \n{x}_1 \leq \n{x}_X \leq \n{x}_1,
\end{equation}
where $\rho = 2^{-\sum_{k=1}^\infty \br{1-\frac{1}{p_k}}} > 0$ by  the choice of
$\{p_n\}_{n=1}^\infty$~\eqref{p_n_condition_smoothness}.
Hence, the $\n{\cdot}_X$-norm is equivalent to $\n{\cdot}_1$-norm, 
and $X = (\ell_1, \n{\cdot}_X)$ is a Banach space.
We note that while we impose condition~\eqref{p_n_condition_smoothness} to obtain
the norms equivalence, the weaker restrictions on the rate of decay of 
$\{p_n\}_{n=1}^\infty$ might be used (see Proposition~1 
from~\cite{dowling_johnson_lennard_turett_ojdt}).

Next, we show that the constructed space $X$ is smooth. Namely, we prove that 
for any element $x \in X$ there is a unique norming functional $F_x$.

First, we establish the dual of $X$.
Let $\{e^*_n\}_{n=1}^\infty$ be the canonical basis in $\ell_\infty$.
Consider the sequence of numbers $\{q_n\}_{n=1}^\infty$ given by 
\[
q_n = \frac{p_n}{p_n-1}.
\]
Similarly, we define the sequence of functionals $\{\nu_n\}_{n=0}^\infty$ as follows:
for any sequence ${a = \sum_{n=1}^\infty a_n e_n^* \in \ell_\infty}$
\begin{align*}
\nu_0(a) &= 0,
\text{ and for any } n \geq 1
\\
\nu_n(a) &= \br{\nu_{n-1}^{q_n}(a) + |a_n|^{q_n}}^{1/q_n}.
\end{align*}
Consider the space 
\[
X^* = \{a \in \ell_\infty : \lim_{n\to\infty} \nu_n(a) < \infty\},
\]
equipped with the norm 
\[
\n{a}_{X^*} = \lim_{n\to\infty} \nu_n(a).
\]
In the same way as above we show that $\n{\cdot}_{X^*}$-norm and
$\n{\cdot}_\infty$-norm are equivalent. For any $n \geq 1$
\[
\nu_n(a) \geq \sup_{k \leq n} |a_k|, 
\]
and
\begin{align*}
\nu_n(a) &= \br{\nu_{n-1}^{q_n}(a) + |a_n|^{q_n}}^{1/q_n}
\leq 2^{\frac{1}{q_n}} \max\{\nu_{n-1}(a), |a_n|\}
\\
&\leq 2^{\frac{1}{q_{n-1}} + \frac{1}{q_n}} \max\{\nu_{n-2}(a), |a_{n-1}|, |a_n|\}
\\
&\cdots
\\
&\leq 2^{\sum_{k=3}^n \frac{1}{q_k}} \max\{\nu_2(a), |a_3|, \ldots, |a_n|\}
\\
&\leq 2^{\sum_{k=2}^n \frac{1}{q_k}} \max\{|a_1|, |a_2|, \ldots, |a_n|\}.
\end{align*}
Therefore, by taking the limit by $n \to \infty$, we obtain for any $a \in X^*$
\[
\n{a}_\infty \leq \n{a}_{X^*} \leq \rho^{-1}\n{a}_\infty,
\]
i.e. the $\n{\cdot}_{X^*}$-norm is equivalent to $\n{\cdot}_\infty$-norm, 
and $X^* = (\ell_\infty, \n{\cdot}_{X^*})$ is a Banach space.

We claim that $X^*$ is the dual of $X$. Indeed, for any $x \in X$ and any $a \in X^*$ 
the H\"older's inequality provides for any $N \in \mathbb{N}$
\begin{align*}
\sum_{n=1}^N |a_n| |x_n|
&\leq \nu_2(a) \vartheta_2(x) + \sum_{n=3}^N |a_n| |x_n|
\\
&\cdots
\\
&\leq \nu_{N-1}(a) \vartheta_{N-1}(x) + |a_N| |x_N|
\leq \nu_N(a) \vartheta_N(x),
\end{align*}
and therefore
\[
|a(x)| = \lim_{N\to\infty} \left| \sum_{n=1}^N a_n x_n \right|
\leq \lim_{N\to\infty} \sum_{n=1}^N |a_n| |x_n|
\leq \n{a}_{X^*} \n{x}_X.
\]
Similarly, using induction we obtain for any functional 
$a(x) = \sum_{n=1}^\infty a_j x_j$ on $X$
\[
\sup_{x \in S_X} a(x) = \n{a}_{X^*},
\]
which completes the proof of the claim.

Consider the spaces $X^n = (\ell_1^n, \vartheta_n(\cdot))$
and ${X^*}^n = (\ell_\infty^n, \nu_n(\cdot))$ --- the initial segments of 
$X$ and $X^*$ respectively.
We use induction to show that for any $n \geq 1$ the space ${X^*}^n$ is strictly convex. 
Indeed, ${X^*}^1 = (\mathbb{R}, |\cdot|)$ is strictly convex, and for any $n > 1$
\[
{X^*}^n = {X^*}^{n-1} \oplus_{q_n} \mathbb{R}
\]
is strictly convex as a $q_n$-sum of strictly convex spaces with 
$1 < q_n < \infty$ (see, e.g.,~\cite{beauzamy_ibstg}).
Therefore $X^n$ is smooth as a predual of a strictly convex space ${X^*}^n$
(e.g.~\cite{beauzamy_ibstg}).
\\
Lastly, we will need the following simple lemma to prove smoothness of $X$.
\begin{lemma}\label{norming_functional_in_initial_segment}
Let $x = \sum_{n=1}^\infty x_n e_n$ be an element in $X$ and 
$F_x = \sum_{n=1}^\infty a_n e^*_n$ be a norming functional for $x$.
Then for any $m \in \mathbb{N}$
\[
F^m_x = \frac{\sum_{n=1}^m a_n e^*_n}{\nu_m(a)}
\]
is a norming functional for $x^m = \sum_{n=1}^m x_n e_n \in X^m$.
\end{lemma}
\begin{proof}
Assume that ${F_x^m(x^m) < \n{x^m}_{X^m} = \vartheta_m(x)}$, i.e. 
$F_x^m$ is not a norming functional for $x^m$. Then 
\[
\sum_{n=1}^m a_n x_n < \nu_m(a) \vartheta_m(x)
\]
and for any $N > m$ by H\"older's inequality
\begin{align*}
F_x(x) &= \sum_{n=1}^\infty a_n x_n 
\leq \sum_{n=1}^\infty |a_n| |x_n|
\\
&< \nu_m(a) \vartheta_m(x) + \sum_{n=m+1}^\infty |a_n| |x_n|
\\
&\leq \nu_N(a) \vartheta_N(x) + \sum_{n=N+1}^\infty |a_n| |x_n|.
\end{align*}
Taking the limit as $N \to \infty$ we get
\[
F_x(x) < \n{a}_{X^*} \n{x}_X = \n{x}_X,
\]
which contradicts $F_x(x) = \n{x}$.
\end{proof}

\begin{lemma}\label{X_is_smooth}\footnote{This proof is due to S.J.\,Dilworth.}
The space $X = (\ell_1, \n{\cdot}_X)$ is smooth.
\end{lemma}
\begin{proof}
Assume that there is an element $x \in X$ with two distinct norming functionals: 
$F_x = \sum_{n=1}^\infty a_n e^*_n$ and $G_x = \sum_{n=1}^\infty b_n e^*_n$.
Then Lemma~\ref{norming_functional_in_initial_segment} and the smoothness of 
initial segments provide for any $N \in \mathbb{N}$
\[
\frac{\sum_{n=1}^N a_n e^*_n}{\nu_N(a)} = F^N_x 
= G^N_x = \frac{\sum_{n=1}^N b_n e^*_n}{\nu_N(b)}.
\]
Find such $m \in \mathbb{N}$ that $a_m \ne b_m$.
Then, taking the limit as $N \to \infty$ and taking into account that
$\n{a}_{X^*} = \n{b}_{X^*} = 1$, we get
\[
a_m = \lim_{N\to\infty} \frac{a_m}{\nu_N(a)} = \lim_{N\to\infty} F_x^N(e_m)
= \lim_{N\to\infty} G_x^N(e_m) = \lim_{N\to\infty} \frac{b_m}{\nu_N(b)} = b_m,
\]
which contradicts $a_m \neq b_m$ and thus $X$ is smooth.
\end{proof}

Finally, we need to establish the norming functionals in $X$.
Take any element $x = \sum_{n=1}^\infty x_n e_n$ in $X$ and
consider a sequence of functionals $\{\mathcal{F}_x^n\}_{n=0}^\infty$ defined as 
follows: for any $y = \sum_{n=1}^\infty y_n e_n \in X$
\begin{align*}
\mathcal{F}_x^0(y) &= 0,
\text{ and for any } n \geq 1
\\
\mathcal{F}_x^n(y) &= \frac{\vartheta_{n-1}^{p_n-1}(x) \mathcal{F}_x^{n-1}(y)  
+ \text{sgn}\,x_n |x_n|^{p_n-1} y_n}{\vartheta_n^{p_n-1}(x)}
\\
&= \vartheta_n^{1-p_{n+1}}(x) \sum_{k=1}^n \br{\text{sgn}\,x_k |x_k|^{p_k-1} y_k 
\prod_{j=k}^n \vartheta_j^{p_{j+1}-p_j}(x)}.
\end{align*}
\begin{lemma}\label{norming_functional_in_X^n}
Let $x = \sum_{n=1}^m x_n e_n$ be an element in $X$.
Then $\mathcal{F}_x^m$ is the norming functional for $x$.
\end{lemma}
\begin{proof}
We will use induction by $m$.
For $m = 1$ 
\[
\mathcal{F}_x^1(y) = \text{sgn}\,x_1 \, y_1,
\]
and $\mathcal{F}_x^1(x) = \vartheta_1(x) = \n{x}_X$, 
$|\mathcal{F}_x^1(y)| = \vartheta_1(y) = \n{y}_X$.
For $m > 1$
\[
\mathcal{F}_x^m(y) = \frac{\vartheta_{m-1}^{p_m-1}(x) \mathcal{F}_x^{m-1}(y)  
+ \text{sgn}\,x_m |x_m|^{p_m-1} y_m}{\vartheta_m^{p_m-1}(x)}.
\]
Then
\[
\mathcal{F}_x^m(x) = \frac{\vartheta_{m-1}^{p_m}(x) + |x_m|^{p_m}}{\vartheta_m^{p_m-1}(x)}
= \vartheta_m(x) = \n{x}_X,
\]
and induction hypothesis and H\"older's inequality provide
\begin{align*}
|\mathcal{F}_x^m(y)| 
&\leq \frac{\vartheta_{m-1}^{p_m-1}(x) |\mathcal{F}_x^{m-1}(y)| 
+ |x_m|^{p_m-1} |y_m|}{\vartheta_m^{p_m-1}(x)}
\\
&\leq \frac{\vartheta_{m-1}^{p_m-1}(x) \vartheta_{m-1}(y) 
+ |x_m|^{p_m-1} |y_m|}{\vartheta_m^{p_m-1}(x)}
\\
&\leq \br{\vartheta_{m-1}^{p_m}(y) + |y_m|^{p_m}}^{1/p_m}
= \vartheta_n(y) = \n{y}_X.
\end{align*}
\end{proof}
Thus, we have established the norming functionals $\mathcal{F}_n$
in the initial segments $X^n$. In particular, for any $x, y \in X$
\begin{align*}
\mathcal{F}_x^1(y) &= \text{sgn}\,x_1 \, y_1,
\\
\mathcal{F}_x^2(y) &= \frac{\text{sgn}\,x_1 |x_1|^{p_2-1} y_1 
+ \text{sgn}\,x_2 |x_2|^{p_2-1} y_2}{\vartheta_2^{p_2-1}(x)},
\\
\mathcal{F}_x^3(y) &= \frac{\br{\text{sgn}\,x_1 |x_1|^{p_2-1} y_1 
+ \text{sgn}\,x_2 |x_2|^{p_2-1} y_2}\vartheta_2^{p_3-p_2}(x)
+ \text{sgn}\,x_3 |x_3|^{p_3-1} y_3}{\vartheta_3^{p_3-1}(x)}.
\end{align*}

We now choose a dictionary $\D$ in $X$ and an element $f \in X$ such that
WCGA of $f$ diverges. Without loss of generality assume $t_n = 1$ for 
each $n \geq 1$, i.e. an element of the dictionary that maximizes
$F_{f_{n-1}}$ is chosen on each step. Let
\begin{align*}
g_0 &= e_1 + e_2 + e_3,
\\
g_k &= e_k + e_{k+1}
\text{ for each } k \geq 1,
\end{align*}
and take $\D = \{ \pm g_n/\n{g_n}_X\}_{n=0}^\infty$.
Note that for any $k \geq 1$
\begin{equation}\label{norm(g_0)>norm(g_k)}
\n{g_k}_X = 2^{1/p_{k+1}} \leq 2^{1/p_2} 
< \br{1 + 2^{p_3/p_2}}^{1/p_3} = \n{g_0}_X.
\end{equation}
Take $f = e_1 \in X$, then $f = g_0 - g_2 \in A_0(\D)$.
We will show that the WCGA diverges even for such a simple element. 
We claim that for any $m \geq 1$
\begin{equation}\label{induction_hypothesis_smoothness_non-sufficiency_wcga}
\phi_m = \pm g_m/\n{g_m}_X,
\end{equation}
where by $\pm$ we understand some sign --- plus or minus.
We will prove this claim using induction by $m$.
\\
Consider the case $m = 1$. Lemma~\ref{norming_functional_in_X^n} provides
$F_f = \mathcal{F}_f^1$, thus
\begin{align*}
&|\mathcal{F}^1_f(g_0)| = 1,
\\
&|\mathcal{F}^1_f(g_1)| = 1,
\\
&|\mathcal{F}^1_f(g_k)| = 0
\text{ for any } k > 1.
\end{align*}
Then estimate~\eqref{norm(g_0)>norm(g_k)} guarantees that $\phi_1 = \pm g_1/\n{g_1}_X$.
\\
Consider the case $m > 1$. By induction hypothesis the elements 
\[
\pm g_1/\n{g_1}_X, \pm g_2/\n{g_2}_X, \ldots, \pm g_{m-1}/\n{g_{m-1}}_X
\]
were chosen on previous steps.
Then $f_{m-1} = \sum_{n=1}^m c_n e_n$ for some coefficients $\{c_n\}_{n=1}^m$, 
and therefore $F_{f_{m-1}} = \mathcal{F}^m_{f_{m-1}}$
by Lemma~\ref{norming_functional_in_X^n}. 
Note that $f_{m-1} \in X^m$, which is a uniformly smooth space since 
it is smooth and finitely-dimensional.
Hence, applying Lemma~\ref{F_n(phi_n)} we obtain 
that $F_{f_{m-1}}(g_k) = 0$ for any $k = 1, \ldots, m-1$, i.e.
\begin{align*}
&\mathcal{F}^m_{f_{m-1}}(g_1) 
= \frac{\text{sgn}\,c_1 |c_1|^{p_2-1} + \text{sgn}\,c_2 |c_2|^{p_2-1}}
{\vartheta_2^{p_2-1}(f_{m-1}) \ldots \vartheta_m^{p_m-1}(f_{m-1})} = 0,
\\
&\mathcal{F}^m_{f_{m-1}}(g_2) 
= \frac{\text{sgn}\,c_2 |c_2|^{p_2-1} \vartheta_2^{p_3-p_2}(f_{m-1}) 
+ \text{sgn}\,c_3 |c_3|^{p_3-1}}
{\vartheta_3^{p_3-1}(f_{m-1}) \ldots \vartheta_m^{p_m-1}(f_{m-1})} = 0,
\\
&\cdots
\\
&\mathcal{F}^m_{f_{m-1}}(g_{m-1}) 
= \frac{\text{sgn}\,c_{m-1} |c_{m-1}|^{p_{m-1}-1} \vartheta_{m-1}^{p_m-p_{m-1}}(f_{m-1}) 
+ \text{sgn}\,c_m |c_m|^{p_m-1}}{\vartheta_m^{p_m-1}(f_{m-1})} = 0.
\end{align*}
From these equalities we derive
\begin{align*}
&|c_2|^{p_2-1} = |c_1|^{p_2-1},
\\
&|c_3|^{p_3-1} = |c_2|^{p_2-1} \vartheta_2^{p_3-p_2}(f_{m-1}),
\\
&\cdots
\\
&|c_m|^{p_m-1} = |c_{m-1}|^{p_{m-1}-1} \vartheta_{m-1}^{p_m-p_{m-1}}(f_{m-1}),
\end{align*}
which imply that for any $k = 3, \ldots, m$
\begin{equation}\label{c_k_recursive}
|c_k|^{p_k-1} = |c_1|^{p_2-1} \prod_{n=2}^{k-1} \vartheta_n^{p_{n+1}-p_n}(f_{m-1}).
\end{equation}
Therefore
\begin{align*}
&|\mathcal{F}^m_{f_{m-1}}(g_0)| 
= |\mathcal{F}^m_{f_{m-1}}(g_0-g_1)|
= \vartheta_m^{1-p_{m+1}}(f_{m-1}) \br{|c_3|^{p_3-1} \prod_{j=3}^m \vartheta_j^{p_{j+1}-p_j}(f_{m-1})},
\\
&|\mathcal{F}^m_{f_{m-1}}(g_m)| 
= \frac{|c_m|^{p_m-1}}{\vartheta_m^{p_m-1}(f_{m-1})},
\\
&|\mathcal{F}^m_{f_{m-1}}(g_k)| = 0
\text{ for any } k \in \mathbb{N} \setminus \{m\}.
\end{align*}
Thus, by~\eqref{c_k_recursive}
\[
|\mathcal{F}^m_{f_{m-1}}(g_0)| 
= \vartheta_m^{1-p_{m+1}}(f_{m-1}) \br{|c_1|^{p_2-1} \prod_{j=2}^m \vartheta_j^{p_{j+1}-p_j}(f_{m-1})}
= |\mathcal{F}^m_{f_{m-1}}(g_m)|, 
\]
and estimate~\eqref{norm(g_0)>norm(g_k)} guarantees that $\phi_m = \pm g_m/\n{g_m}_X$,
which completes the proof of 
assumption~\eqref{induction_hypothesis_smoothness_non-sufficiency_wcga}.

Hence, the element $\pm g_0/\n{g_0}_X$ will not be chosen
and $\Phi_n = \text{span}\,\{g_1, \ldots, g_n\}$ for any $n \geq 1$.
Then the equivalence of the norms~\eqref{1_X_norm_equivalence} provides
\[
\n{f_n}_X = \inf_{G\in\Phi_n} \n{f - G}_X
\geq \rho \inf_{G\in\Phi_n} \n{f - G}_1 = \rho \not\to 0
\text{ as } n \to \infty,
\]
i.e. the WCGA of $f$ diverges.

\section*{Acknowledgements}
I would like to thank V.\,N.~Temlyakov for his guidance,
and S.\,J.~Dilworth for his continued input and advice on Banach space theory.
\\
I am also grateful to the anonymous reviewer who examined this paper in detail
and whose valuable remarks enhanced the structure of the paper and especially
the statement of Theorem~\ref{theorem_gawcga_convergence_in_P_q}.

\section*{References}
\bibliographystyle{abbrv}
\bibliography{bibliography}

\end{document}